\documentclass[article,11pt]{article}
\usepackage[lmargin=25mm, tmargin=25mm, textwidth=160mm, textheight= 221mm]{geometry}
\usepackage{tikz,multicol}
\usepackage{amsthm,amsmath,amssymb, amsfonts}
\usepackage{graphicx}
\usepackage[colorlinks=true,citecolor=black,linkcolor=black,urlcolor=blue]{hyperref}
\usepackage{enumerate}
\usepackage{geometry}                
\usepackage{xcolor}

\theoremstyle{plain}
\newtheorem{theorem}{Theorem}
\newtheorem{lemma}[theorem]{Lemma}

\theoremstyle{definition}

\newtheorem{conjecture}[theorem]{Conjecture}

\theoremstyle{remark}

\newcommand{\z}{\mathbb{Z}}
\newcommand{\aw}{\operatorname{aw}}
\newcommand{\awu}{\operatorname{aw}_u}

\title{Anti-van der Waerden numbers of 3-term arithmetic progressions.}

\author{Zhanar Berikkyzy\thanks{Department of Mathematics, Iowa State University, Ames, IA 50011, USA (zhanarb@iastate.edu)}, Alex Schulte\thanks{Department of Mathematics, Iowa State University, Ames, IA 50011, USA (aschulte@iastate.edu)}, and Michael Young\thanks{Department of Mathematics, Iowa State University, Ames, IA 50011, USA (myoung@iastate.edu)}}

\begin{document}
	\maketitle
	
\begin{abstract}
The \emph{anti-van der Waerden number}, denoted by $\aw([n],k)$, is the smallest $r$ such that every exact $r$-coloring of $[n]$ contains a rainbow $k$-term arithmetic progression. Butler et. al. showed that $\lceil \log_3 n \rceil + 2 \le \aw([n],3) \le \lceil \log_2 n \rceil + 1$, and conjectured that there exists a constant $C$ such that $\aw([n],3) \le \lceil \log_3 n \rceil + C$. In this paper, we show this conjecture is true by determining $\aw([n],3)$ for all $n$. We prove that for $7\cdot 3^{m-2}+1 \leq n \leq 21 \cdot 3^{m-2}$, 
		\begin{equation*}
			\aw([n],3)=\left\{\begin{array}{ll}
				m+2, & \mbox{if $n=3^m$} \\
				m+3, & \mbox{otherwise}.
			\end{array}\right.
		\end{equation*}
\end{abstract}

\noindent{\bf Keywords.} {arithmetic progression; rainbow coloring; unitary coloring; 
Behrend construction.}
	
	\section{Introduction}
	
	Let $n$ be a positive integer and let $G \in \{[n], \z_n\}$, where $[n]=\{1,\ldots,n\}$. A \emph{$k$-term arithmetic progression} ($k$-AP) of $G$ is a sequence in $G$ of the form $$a, a+d, a + 2d, \ldots, a+(k-1)d,$$ where $d\geq 1$. For the purposes of this paper, an arithmetic progression is referred to as a set of the form $\{a, a+d, a + 2d, \ldots, a+(k-1)d \}$. An \emph{$r$-coloring} of  $G$ is a function $c: G \rightarrow [r]$, and such a coloring is called \emph{exact} if $c$ is surjective. Given $c: G \rightarrow [r]$, an arithmetic progression is called {\em rainbow} (under $c$) if $c(a + id) \neq c(a+jd)$ for all $0 \le i < j \le k-1$. 
	
	The \emph{anti-van der Waerden number}, denoted by $\aw(G,k)$, is the smallest $r$ such that every exact $r$-coloring of $G$ contains a rainbow $k$-AP. If $G$ contains no $k$-AP, then $\aw(G,k)= |G|+1$; this is consistent with the property that there is a coloring of $G$ with $\aw(G,k)-1$ colors that has no rainbow $k$-AP.
	
	An $r$-coloring of $G$ is \emph{unitary} if there is an element of $G$ that is uniquely colored. The smallest $r$ such that every exact unitary $r$-coloring of $G$ contains a rainbow $k$-AP is denoted by $\awu(G,k)$. Similar to the anti-van der Waerden number, $\awu(G,k)=|G|+1$ if $G$ has no $k$-AP.

Problems involving counting and the existence of rainbow arithmetic progressions have been well-studied. The main results of Axenovich and Fon-Der-Flaass \cite{AF} and Axenovich and Martin \cite{AM} deal with the existence of 3-APs in colorings that have uniformly sized color classes. Fox, Jungi\'c, Mahdian, Ne\u{s}etril, and Radoi\u{c}i\'c also studied anti-Ramsey results of arithmetic progressions in \cite{J}. In particular, they showed that every 3-coloring of $[n]$ for which each color class has density more than $1/6$, contains a rainbow 3-AP. Fox et. al. also determined all values of $n$ for which $\aw(\z_n,3) = 3$. 

The specific problem of determining anti-van der Waerden numbers for $[n]$ and $\z_n$ was studied by Butler et. al. in \cite{awpaper}. It is proved in \cite{awpaper} that for $k \ge 4$, $\aw([n],k) = n^{1- o(1)}$ and $\aw(\z_n, k) = n^{1 - o(1)}$. These results are obtained using results of Behrend \cite{B} and Gowers \cite{G} on the size of a subset of $[n]$ with no $k$-AP. Butler et. al. also expand upon the results of \cite{J} by determining $\aw(\z_n,3)$ for all values of $n$. These results were generalized to all finite abelian groups in \cite{Y}. Butler et. al. also provides bounds for $\aw([n],3)$, as well as many exact values (see Table \ref{tab:awtable}).

\begin{table}[htp]
\centering
{\small
\begin{tabular}[h]{c|ccccccccccccc}
$n\setminus k$ & 3 & 4 & 5 & 6 & 7 & 8 & 9 & 10 & 11 & 12 & 13 & 14 \\\hline
3 & 3\\
4 & 4 \\
5 & 4 & 5\\
6 & 4 & 6 \\
7 & 4 & 6 & 7 \\
8 & 5 & 6 & 8 \\
9 & 4 & 7 & 8 & 9\\
10 & 5 & 8 & 9  & 10 \\
11 & 5 & 8 & 9 & 10 & 11\\
12 & 5 & 8 & 10 & 11 & 12\\
13 & 5 & 8 & 11 &11 & 12 & 13\\
14 & 5 & 8 & 11 & 12 & 13 & 14\\
15 & 5 & 9 & 11 & 13 &14 & 14 & 15\\
16 & 5 & 9 & 12 & 13 & 15 & 15& 16\\
17 & 5 & 9 & 13 & 13& 15 & 16 & 16 & 17\\
18 & 5 & 10 & 14 & 14 & 16 & 17 & 17 & 18 \\
19 & 5 & 10 & 14 & 15 & 17 & 17 & 18 & 18  & 19 \\
20 & 5 & 10 & 14 & 16 & 17 & 18  & 19 & 19 & 20\\
21 & 5 & 11 & 14 & 16 & 17 & 19 & 20 & 20 & 20 & 21 \\
22 & 6 & 12 & 14 & 17 & 18 & 20 & 21 & 21 & 21 & 22\\
23 &  6 & 12 & 14 & 17 & 19 & 20 & 21 & 22 &22 & 22 & 23\\
24 & 6 & 12 & 15 & 18 & 20 & 20 & 22 & 23 & 23 & 23 &24 \\
25 & 6 & 12 & 15 & 19 & 21 & 21 & 23 & 23 & 24 & 24 & 24 & 25  \\
\end{tabular}
}
\caption{\label{tab:awtable}Values of $\aw([n],k)$ for $3 \leq k \leq \frac{n+3}{2}$.}
\end{table}

In this paper, we determine the exact value of $\aw([n],3)$, which answers questions posed in \cite{awpaper} and confirms the following conjecture:

\begin{conjecture}\cite{awpaper}
There exists a constant $C$ such that $\aw([n],3) \le \lceil \log_3 n \rceil + C$, for all $n \ge 3$.
\end{conjecture}

Our main result, Theorem \ref{maintheorem}, also determines $\awu([n],3)$ which shows the existence of extremal colorings of $[n]$ that are unitary. 

		\begin{theorem} \label{maintheorem}
		For all integers $n \ge 2$,
		\begin{equation*}
			\awu([n],3) = \aw([n],3)=\left\{\begin{array}{ll}
				m+2, & \mbox{if $n=3^m$} \\
				m+3, & \mbox{if $n\not=3^m$ and $7\cdot 3^{m-2}+1\leq n\leq 21\cdot 3^{m-2}$}.
			\end{array}\right.
		\end{equation*}
	\end{theorem}
	
In section \ref{sec:lem}, we provide lemmas that are useful in proving Theorem \ref{maintheorem} and section \ref{sec:pf} contains the proof of Theorem \ref{maintheorem}.

	\section{Lemmas}\label{sec:lem}

In \cite[Theorem 1.6]{awpaper} it is shown that $3\leq \aw(\z_p,3)\leq 4$ for every prime number $p$ and that if $\aw(\z_p,3)= 4$ then $p\geq 17$. Furthermore, it is shown that the value of $aw(\z_n,3)$ is determined by the values of $aw(\z_p,3)$ for the prime factors $p$ of $n$. We have included this theorem below with some notation change.

	\begin{theorem}\label{LH:Zn3primefactor3} \cite{awpaper}
		Let $n$ be a positive integer with prime decomposition $n=2^{e_0}p_1^{e_1}p_2^{e_2}\cdots p_s^{e_s}$ for $e_i\geq 0$, $i=0,\ldots,s$, where primes are ordered so that $\aw(\z_{p_i},3)=3$ for $ 1 \leq i \leq \ell$ and $\aw(\z_{p_i},3)=4$ for $\ell + 1 \leq i \leq s$. Then 
		

		\begin{equation*}
			\aw(\z_n,3)=\left\{\begin{array}{ll}
				2 +\sum\limits_{j=1}^\ell e_j + \sum\limits_{j=\ell+1}^s 2e_j, & \mbox{if $n$ is odd} \\
				3 +\sum\limits_{j=1}^\ell e_j + \sum\limits_{j=\ell+1}^s 2e_j, & \mbox{if $n$ is even}.
			\end{array}\right.
		\end{equation*}

	\end{theorem}

	We use Theorem \ref{LH:Zn3primefactor3} to prove the following lemma.
	
	\begin{lemma}\label{ZnLessLog}
		Let $n \geq 3$, then $\aw(\z_n,3) \leq \lceil\log_3 n \rceil + 2$ with equality if and only if $n= 3^j$ or $2\cdot 3^j$ for $j\geq 1$. 
	\end{lemma}
	
	\begin{proof}
		
		Suppose $n=2^{e_0}p_1^{e_1}p_2^{e_2}\ldots p_s^{e_s}$ with $e_i \geq 0$ for $i=0,\dots,s$, where primes $p_1, p_2, \ldots, p_s$ are ordered so that $\aw(\z_{p_i},3)=3$ for $1 \leq i\leq \ell$ and $\aw(\z_{p_i},3)=4$ for $\ell + 1 \leq i\leq s$. We consider two cases depending on parity of $n$.

		\textit{Case 1.} Suppose $n$ is odd, that is $e_0=0$. Then $\aw(\z_n,3)= 2 +  \sum\limits_{j=1}^\ell e_j + \sum \limits_{j=\ell+1}^s 2e_j$ by Theorem~\ref{LH:Zn3primefactor3}. Since $\aw(\z_p,3)=3$ for odd primes $p\leq 13$, we have $p_i\geq 17$ for $i\geq \ell+1$, and clearly $p_i\geq 3$ for $i \leq \ell$, therefore  
		
		$$3^{\aw(\z_n,3)} = 3^{2+\sum\limits_{j=1}^\ell e_j + \sum\limits_{j=\ell +1}^s 2e_j}= 9\cdot 3^{e_1}\cdots 3^{e_{\ell}}\cdot 9^{e_{\ell+1}}\cdots 9^{e_s}\leq 9\cdot p_1^{e_1}\cdots p_s^{e_s}= 9n.$$

		Note that the equality holds if and only if $n$ is a power of $3$, that is $e_j=0$ for $2 \leq j \leq s$. Therefore, $\aw(\z_n,3)\leq \lceil  \log_3 n\rceil + 2$ for odd $n$, with equality if and only if $n=p_1^{e_1}$. \\
		
		\textit{Case 2.} Suppose $n$ is even, that is $e_0\geq 1$. Then $\aw(\z_n,3)= 3 +  \sum\limits_{j=1}^\ell e_j + \sum \limits_{j=\ell+1}^s 2e_j$ by Theorem~\ref{LH:Zn3primefactor3}. If $n=2^{e_0}\cdot  3^j$ for $j\geq 1$, then by direct computation $\aw(\z_n,3)=3 + j \le 2 + \lceil\log_3 n \rceil$, with equality if and only if $e_0 = 1$. So suppose there is $i$ such that $p_i\not=3$, and let $h=\frac{n}{2^{e_0}p_i^{e_i}}$.

		If $i\leq \ell$ then $p_i\geq 5$, and so $3\cdot 3^{e_i} < 2^{e_0}p_i^{e_i}$ for all $e_0\geq 1$ and $e_i\geq 1$. Therefore, since $h$ is odd, by the previous case
		
		$$3^{\aw(\z_n,3)} = 3\cdot 3^{e_i}\cdot 3^{\aw(\z_h,3)}\leq 3\cdot 3^{e_i}\cdot 9h < 2^{e_0}p_i^{e_i}\cdot 9h =9n.$$ 
		
		If $i\geq \ell+1$ then $p_i\geq 17$, and so $3\cdot 9^{e_i} < 2^{e_0}p_i^{e_i}$ for all $e_0\geq 1$ and $e_i\geq 1$. Then by the previous case
		
		$$3^{\aw(\z_n,3)} = 3\cdot 9^{e_i}\cdot 3^{\aw(\z_h,3)}\leq 3\cdot 9^{e_i}\cdot 9h < 2^{e_0}p_i^{e_i}\cdot 9h =9n.$$
		
	\end{proof}

	A set of consecutive integer $I$ in $[n]$ is called an \emph{interval} and $\ell(I)$ is the number of integers in $I$. Given a coloring $c$ of some finite nonempty subset $S$ of $[n]$, a \emph{color class} of a color $i$ under $c$ in $S$ is denoted $c_{i}(S) :=\{x\in S : c(x)=i\}$. A coloring $c$ of $[n]$ is \emph{special} if $n=7q +1$ for some positive integer $q$, $c(1)$ and $c(n)$ are both uniquely colored, and there are two colors $\alpha$ and $\beta$ such that $c_{\alpha}([n]) = \{q+1, 2q+1,4q+1\}$ and $c_{\beta}([n]) = \{3q+1,5q+1,6q+1\}$.

	\begin{lemma}\label{special_coloring}
		Let $N$ be an integer and $c$ be an exact $r$-coloring of $[N]$ with no rainbow $3$-AP, where $1$ and $N$ are colored uniquely. Then either the coloring $c$ is special or $|\{c(x): x\equiv i \pmod 3 \mbox{ and } x\in [N] \}|\geq r-1$  for $i=1$ or $i=N$. 
	\end{lemma}
	
	\begin{proof}
		Observe that $N$ is even, otherwise $\{1,(N+1)/2,N\}$ is a rainbow $3$-AP. We partition the interval $[N]$ into four subintervals $I_1=\{1, \ldots, \lceil N/4 \rceil\}$, $I_2=\{\lceil N/4 \rceil +1, \ldots, N/2\}$, $I_3=\{N/2+1,\ldots, \lfloor 3N/4 \rfloor\}$, and $I_4=\{\lfloor 3N/4 \rfloor +1, \ldots, N\}$. Notice that every color other than $c(1)$ and $c(N)$ must be used in the subinterval $I_2$. To see this, assume $i$ is the missing color in $I_2$ distinct from $c(1)$ and $c(N)$. Let $x$ be the largest integer in $c_{i}(I_1)$. Since $N$ is even, we have $2x-1\leq 2\lceil N/4 \rceil -1\leq N/2$, and so $2x-1\in I_2$ and $c(2x-1)\not=i$. Therefore the $3$-AP $\{1,x,2x-1\}$ is a rainbow. If there is no such integer $x$ in $I_1$, then the integers colored with $i$ must be in the second half of the interval $[N]$, so we choose the smallest such integer $y$ in $c_{i}(I_3 \cup I_4)$. Then $\{2y-N, y, N\}$ is a rainbow $3$-AP since $c(2y-N)\not=i$, because $2y-N\in I_1\cup I_2$. Similarly, every color other than $c(1)$ and $c(N)$ must be used in the subinterval $I_3$. 
		
		Throughout the proof we mostly drop $\mkern-8mu\pmod 3$ and just say congruent even though we mean congruent modulo $3$. We consider the following three cases. 
		
		
		\textit{Case 1:} $N\equiv 0 \pmod 3$. Assume $|\{c(x): x\equiv i \pmod 3 \mbox{ and }  x\in [N]\}|< r-1$  for both $i=1$ and $i=N$. So there are two colors, say $red$ and $blue$, such that no integer in $[N]$ colored with $red$ is congruent to $1$, and no integer in $[N]$ colored with $blue$ is congruent to $0$. We further partition the interval $I_2$ into subintervals $I_{2(i)}$ and $I_{2(ii)}$ so that $\ell(I_{2(i)})\leq \ell(I_{2(ii)})\leq \ell(I_{2(i)})+1$, and partition the interval $I_3$ into subintervals $I_{3(i)}$ and $I_{3(ii)}$ so that $\ell(I_{3(ii)})\leq \ell(I_{3(i)})\leq \ell(I_{3(ii)})+1$. Then we have the following observations:

		\vspace{0.2cm}  
		
		\noindent \textit{(i)} $x\equiv 0$ for all $x \in c_{red}(I_3 \cup I_4)$ and $y\equiv 1$ for all $y\in c_{blue}(I_1 \cup I_2)$.
		
		\noindent If there is an integer $r$ in $I_3 \cup I_4$ colored with $red$ and congruent to $2$, then $2r-N\equiv 1$, and so $c(2r-N)$ is not $red$ by our assumption. Therefore the $3$-AP $\{2r-N, r, N\}$ is rainbow. Similarly, if there is an integer $b$ in $I_1 \cup I_2$ colored with $blue$ and congruent to $2$, then $2b-1\equiv 0$, and so $c(2b-1)$ is not $blue$, forming a rainbow $3$-AP $\{1, b, 2b-1\}$. 
		\vspace{0.2cm}
		
		\noindent \textit{(ii)} $x\equiv 2$ for all $x \in c_{red}(I_2)$ and $y\equiv 2$ for all $y\in c_{blue}(I_3)$.
		
		\noindent  If there is an integer $r$ in $c_{red}(I_2)$ congruent to $0$, then $2r-1\equiv 2$ and $2r-1\in I_3\cup I_4$ since $2r-1\geq N/2+1$. Therefore, $2r-1$ is not colored with $red$ by the previous observation, and so the $3$-AP $\{1,r,2r-1\}$ is a rainbow. Similarly, if there is an integer $b$ in $c_{blue}(I_3)$ congruent to $1$, then using $N$ we obtain the rainbow $3$-AP $\{2b-N,b,N\}$, because $2b-N\equiv 2$ and $2b-N\leq N/2$.
		\vspace{0.2cm}
		
		\noindent \textit{(iii)} $c_{red}(I_{3(ii)})=c_{blue}(I_{2(i)})=\emptyset$.
		
		\noindent  If there is an integer $r$ in $I_{3(ii)}$ colored with $red$, then $2r-N\equiv 0$, by observation \textit{(i)}. Furthermore, $2r-N\leq N/2$ and $2r-N\geq 2(N/2+\ell(I_{3(i)})+1)-N\geq (2\ell(I_{3(i)})+1)+1\geq \lceil N/4\rceil +1$. So $2r-N \in I_2$ and hence it is not colored with $red$ by observation $(ii)$. Therefore, $\{2r-N,r,N\}$ is a rainbow $3$-AP. Similarly, if there is an integer $b$ in $I_{2(i)}$ colored with $blue$, then $2b-1\equiv 1$ and $N/2+1\leq 2b-1\leq \lfloor 3N/4\rfloor$. So $2b-1 \in I_3$ and hence it is not colored with $blue$ by observation $(ii)$. Therefore, $\{1,b,2b-1\}$ is a rainbow $3$-AP. 
		\vspace{0.2cm}
		
		\noindent \textit{(iv)} $c_{red}(I_{2(ii)})=c_{blue}(I_{3(i)})=\emptyset$.
		
		\noindent  Suppose there is an integer $r$ in $I_{2(ii)}$ colored with $red$. Since the coloring of $I_2$ contains both $red$ and $blue$ and there is no integer in $I_{2(i)}$ colored with $blue$, by $(iii)$, there must be an integer $b$ in $I_{2(ii)}$ colored with $blue$. By $(i)$ and $(ii)$, $b\equiv 1$ and $r\equiv 2$. Wlog, suppose $b>r$. Then $2r-b\equiv 0$ and $2r-b\in I_2$ since $\ell(I_{2(ii)})\leq \ell(I_{2(i)})+1$. So $2r-b$ is not colored $red$ or $blue$ and hence the $3$-AP $\{2r-b,r,b\}$ is rainbow. Therefore, there is no integer in $I_{2(ii)}$ that is colored with $red$. Similarly, there is no integer in $I_{3(i)}$ that is colored with $blue$. 
		
		\vspace{0.2cm}
		Recall that every color other than $c(1)$ and $c(N)$ is used in both intervals $I_2$ and $I_3$. Therefore, sets $c_{red}(I_{2(i)})$, $c_{blue}(I_{2(ii)})$, $c_{red}(I_{3(i)})$, and $c_{blue}(I_{3(ii)})$ are nonempty. Using above observations we next show that in fact these integers colored with $blue$ and $red$ in each subinterval are unique. Let $B=\{b_1,\ldots, b_2\}$ be the shortest interval in $I_{2(ii)}$ which contains all integers colored with $blue$ and let $R=\{r_1,\ldots, r_2\}$ be the shortest interval in $I_{3(i)}$ which contains all integers colored with $red$. Choose the largest integer $x$ in $c_{red}(I_{2(i)})$ and consider two $3$-APs $\{x,b_1,2b_1-x\}$ and $\{x,b_2,2b_2-x\}$. Since $x$ is congruent to $2$ and both $b_1$ and $b_2$ are congruent to $1$, we have that both $2b_1-x$ and $2b_2-x$ are congruent to $0$ and are contained in $I_3$, otherwise the $3$-APs are rainbow. Since all integers colored with $blue$ in $I_3$ are congruent to $2$ by \emph{(ii)}, we have that $2b_1-x$ and $2b_2-x$ are both colored with $red$ and so contained in $R$. Therefore, $2\ell(B)-1\leq \ell(R)$. Now using the smallest integer in $c_{blue}(I_{3(ii)})$, we similarly have that $2\ell(R)-1\leq \ell(B)$. Since $\ell(B)\geq 1$ and $\ell(R)\geq 1$, we have that $\ell(R)=\ell(B)=1$, i.e. there are unique integers $b$ in $c_{blue}(I_{2(ii)})$ and $r$ in $c_{red}(I_{3(i)})$. 
		
		Now for any integer $\tilde{r}$ from $c_{red}(I_{2(i)})$ the integer $2\tilde{r}-1$ must be colored with $red$, otherwise the $3$-AP $\{1,\tilde{r},2\tilde{r}-1\}$ is rainbow. Since $2\tilde{r}-1\in I_3$, it must be equal to the unique $red$ colored integer $r$ of $I_3$. Therefore, there is exactly one such $\tilde{r}$ in $I_{2(i)}$, i.e. $c_{red}(I_{2(i)})=\{\tilde{r}\}$. Similarly, using $N$ there is a unique integer $\tilde{b}$ in $I_{3(ii)}$ colored with $blue$. Since $\{1,\tilde{r},r\}$, $\{\tilde{r},b,r\}$, $\{b,r,\tilde{b}\}$, and $\{b,\tilde{b},N\}$ are all $3$-APs, $N=7(\ell(\{b,\ldots,r\})-1)+1 = 7(r-b)+1$.
		
		Observe that if $\tilde{r}$ is even, the integer $(\tilde{r}+N)/2$ in $3$-AP $\{\tilde{r},(\tilde{r}+N)/2, N\}$ must be $red$ and congruent to $1$ since $\tilde{r}\equiv 2$ by \emph{(ii)}, contradicting our assumption. So $\tilde{r}$ is odd, and hence the integer $r'=(\tilde{r}+1)/2$ in $I_1$ must be colored with $red$. Notice that there cannot be another integer $x$ larger than $r'$ in $c_{red}(I_1)$, otherwise $2x-1$ will be another integer colored with $red$ in $I_2$ distinct from $\tilde{r}$. Now, since $\ell(\{r',\ldots,\tilde{r}\})=\ell(\{b,\ldots,r\})$ we have that $\{r',r,N\}$ is a $3$-AP, and so $r'$ must be even. Suppose there are integers smaller than $r'$ in $c_{red}(I_1)$, and let $z$ be the largest of them. Then $2z-1$ is also in $c_{red}(I_1)$ and must be equal to or larger than $r'$ in $I_1$. However, that is impossible because $r'$ is even and there is no integer in $c_{red}(I_1)$ larger than $r'$. So $r'$ is a unique integer in $I_1$ colored with $red$. Similarly, there is a unique integer $b'$ in $I_4$ colored with $blue$. Therefore the $8$-AP can be formed using integers $1,r',\tilde{r},b,r,\tilde{b},b',N$ since $\ell(\{1,\ldots,r'\})=\ell(\{r',\ldots,\tilde{r}\})=\ell(\{\tilde{r},\ldots,b\})=\ell(\{b,\ldots,r\})=\ell(\{r,\ldots,\tilde{b}\})=\ell(\{\tilde{b},\ldots,b'\})=\ell(\{b',\ldots,N\})$. 
		
		In order for this coloring to be special, it remains to show that $c_{blue}(I_1)=c_{red}(I_4)=\emptyset$. If $c_{blue}(I_1)\not=\emptyset$, then choose the largest integer $y$ in it and consider the $3$-AP $\{1,y,2y-1\}$. Since $2y-1$ must be in $c_{blue}(I_2)$ and the only integer in this set is $b$, we have $2y-1=b$. However, we know that $b$ is even because $b=2\tilde{b}-N$, a contradiction. Similarly, if $c_{red}(I_4)\not=\emptyset$ choose the smallest integer $x$ in it and consider the $3$-AP $\{2x-N,x,N\}$. Since $2x-N$ must be in $c_{red}(I_3)$ and the only integer in this set is $r$, we have $2x-N=r$. However, we know that $r$ is odd because $r=2\tilde{r}-1$, a contradiction. This implies that $c_{red}([N])=\{r',\tilde{r},r\}$ and $c_{blue}([N])=\{b,\tilde{b},b'\}$, so the coloring is special.

		\textit{Case 2:} $N\equiv 2 \pmod 3$. This case is analogous to Case 1.

		\textit{Case 3:} $N\equiv 1 \pmod 3$. Assume $|\{c(x): x\equiv i \pmod 3 \mbox{ and }  x\in [N]\}|< r-1$ i.e. there are two colors, say $red$ and $blue$, such that no integer in $[N]$ colored with $red$ or $blue$ is congruent to $1$. Recall that every color other than $c(1)$ and $c(N)$ appears in $I_2$ and $I_3$. First, notice that all integers colored with $red$ or $blue$ in $I_2$ must be congruent modulo $3$. Otherwise, choosing a $red$ colored integer and a $blue$ colored integer, we obtain a $3$-AP whose third term is colored with $red$ or $blue$ and is congruent to $1$ contradicting our assumption. Similarly, this is also the case for $I_3$. So suppose all integers in $c_{red}(I_2)\cup c_{blue}(I_2)$ and $c_{red}(I_3)\cup c_{blue}(I_3)$ are congruent modulo $3$ to integers $p\not\equiv 1$ and $q\not\equiv 1$, respectively. Pick the largest integers from $c_{red}(I_2)$ and $c_{blue}(I_2)$ and form a $3$-AP whose third term is in $I_3$. Then the third term is colored with $red$ or $blue$ and is congruent to $p$. Therefore, $p\equiv q\not\equiv 1$.
		
		We further partition the interval $I_2$ into subintervals $I_{2(i)}$ and $I_{2(ii)}$, so that $\ell(I_{2(i)})\leq \ell(I_{2(ii)})\leq \ell(I_{2(i)})+1$. If there exists $x\in c_{red}(I_{2(i)})\cup c_{blue}(I_{2(i)})$, the integer $2x-1$ must be colored with $c(x)$ and contained in $I_3$, so $2x-1\equiv p$ while $x\equiv p\not\equiv 1$, a contradiction. So $c_{red}(I_{2(i)})\cup c_{blue}(I_{2(i)})=\emptyset$. However, then the smallest integers of $c_{red}(I_{2(ii)})$ and $c_{blue}(I_{2(ii)})$ form a $3$-AP whose first term is contained in $I_{2(i)}$ and is colored with $red$ or $blue$, a contradiction. This completes the proof of the lemma.
		
	\end{proof}
	
	\section{Proof of Theorem \ref{maintheorem}}\label{sec:pf}

	Given a positive integer $n$, define the function $f$ as follows: 
	\begin{equation*}
		f(n)=\left\{\begin{array}{ll}
			m+2, & \mbox{if $n=3^m$} \\
			m+3, & \mbox{if $n\not=3^m$ and $7\cdot 3^{m-2}+1\leq n\leq 21\cdot 3^{m-2}$}.
		\end{array}\right.
	\end{equation*}
In this section, we prove Theorem \ref{maintheorem} by showing that $\aw([n],3) = f(n)$ for all $n$.

	
		First, we show that $f(n) \leq \awu([n], 3)$ by inductively constructing a unitary coloring of $[n]$ with $f(n) - 1$ colors and no rainbow 3-AP. The result is true for $n = 1,2,3$, by inspection. Suppose $n > 3$ and that the result holds for all positive integers less than $n$. Let $n = 3h-s$, where $s \in \{0, 1, 2 \}$ and $2\leq h < n$. 
		
		Let $r=\awu([h],3)$. So there is an exact unitary $(r-1)$-coloring $c$ of $[h]$	with no rainbow $3$-AP. Let $red$ be a color not used in  $c$. Define the coloring $c_1$ of $[n]$ such that if $x\equiv 1 \pmod 3$, then $c_1(x)=c((x+2)/3)$, otherwise color $x$ with $red$. When $s\not=0$, define the coloring $c_2$ of $[n]$ as follows: if  $x\not\equiv 0 \pmod 3$ then color $x$ with $red$; if $x\equiv 0 \pmod 3$ then $c_2(x)=c(x/3+1)$ when $c(h)$ is the only unique color in $c$ and $c_2(x)=c(x/3)$ otherwise. Notice that $c_2$ is a unitary $\awu([h-1],3)$-coloring when $s\not=0$ and $c_1$ is a unitary $r$-coloring of $[n]$. Now consider a $3$-AP $\{a,b,2b-a\}$ in $[n]$. If $a\equiv b\not\equiv 1$, then $a$ and $b$ are colored with $red$, and so the $3$-AP is not a rainbow. If $a\equiv b\equiv 1$, then $2b-a\equiv 1$, so this set corresponds to a $3$-AP in $[h]$ with coloring $c$, and hence the $3$-AP is not rainbow. If $a\not\equiv b$, then $2b-a$ is not congruent to $a$ or $b$, so two of the terms of the $3$-AP are colored with $red$, and hence the $3$-AP is not rainbow under $c_1$. Similarly, this $3$-AP is not rainbow under $c_2$. Therefore, $c_1$ and $c_2$ are unitary colorings of $[n]$ with no rainbow $3$-AP. Also note that $\awu([n],3) \geq \awu([h],3)+1$ under $c_1$ and $\awu([n],3)\geq\awu([h-1],3)+1$ under $c_2$. We proceed with three cases determined by $\frac{n}{3}$. \\
	
		\textit{Case 1.} First suppose $7\cdot 3^{m-2} + 1 \leq n \leq 3^m-3$ or $3^m \leq n \leq 21\cdot 3^{m-2}$. By the induction hypothesis and using the coloring $c_1$, 
		$$\awu([n],3)\geq \awu([h],3)+1\geq f(h)+1=f(n).$$
		
		\textit{Case 2.} Suppose $n=3^m -t$ where $t \in \{1, 2\}$. Notice that $h=3^{m-1}$, so by induction and using coloring $c_2$,
		
		$$\awu([n],3)\geq \awu([h-1],3)+1 \geq f(h-1)+1 = f(3^{m-1}-1)+1 =(m+2)+1 = f(n).$$


		The upper bound, $\aw([n],3) \le f(n)$, is also proved by induction on $n$. For small $n$, the result follows from Table \ref{tab:awtable}. Assume the statement is true for all value less than $n$, and let $7\cdot 3^{m-2}+1\leq n\leq 21\cdot 3^{m-2}$ for some $m$. Let $aw([n])=r+1$, so there is an exact $r$-coloring $\hat{c}$ of $[n]$ with no rainbow $3$-AP. We need to show that $r\leq f(n)-1$. Let $[n_1,n_2,\ldots,n_N]$ be the shortest interval in $[n]$ containing all $r$ colors under $\hat{c}$. Define $c$ to be an $r$-coloring of $[N]$ so that $c(j)=\hat{c}(n_j)$ for $j\in \{1,\ldots, N\}$. By minimality of $N$ the colors of $1$ and $N$ are unique. If $[N]$ has at least $r-1$ colors congruent to $1$ or $N$, then $[n]$ has at least $r-1$ colors congruent to $n_1$ or $n_N$, respectively, so $r \leq aw(\lfloor n/3 \rfloor)$ and by induction $r\leq f(\lfloor n/3 \rfloor)\leq f(n)-1$. So suppose that is not the case, then by Lemma~\ref{special_coloring} we have that the coloring $c$ is special. 
		
		Let $N=7q+1$ for some $q\geq 1$, and let the $8$-AP in this special coloring be $\{1,r_1,r_2,b_1,r_3,b_2,b_3,N\}$, where $r_1,r_2,r_3$ are the only integers colored $red$, $b_1,b_2,b_3$ are the only integers colored $blue$ and $q=r_1-1$. If $n\geq 9q$, then the $8$-AP can be extended to a $9$-AP in $n$ by adding the $9$th element to either the beginning or the ending. Wlog, suppose $\{1,r_1,r_2,b_1,r_3,b_2,b_3,N, 2N-b_3\}$ correspond to a $9$-AP in $[n]$. Since the coloring has no rainbow $3$-AP, the color of $2N-b_3$ is $blue$ or $c(N)$, so we have a $4$-coloring of this $9$-AP. However, $aw([9],3)=4$ and hence there is a rainbow $3$-AP in this $9$-AP which is in turn a rainbow $3$-AP in $[n]$. Therefore, $n\leq 9q-1$. 
		
		By uniqueness of $red$ colored integer $r_1$ in interval $\{1,\ldots,r_2-1\}$, the colors of integers in interval $\{r_1+1,\ldots,r_2-1\}$ is the same as the reversed colors of integers in $\{2,\ldots,r_1-1\}$, i.e. $c(r_1+i)=c(r_1-i)$ for $i=1,\ldots,q-1$. Similarly, coloring of integers in interval $\{r_2+1,\ldots, b_1-1\}$ is the reversed of the coloring of integers in interval $\{r_1+1,\ldots, r_2-1\}$, and so on. This gives a rainbow $3$-AP-free $(r-2)$-coloring of $\mathbb{Z}_{2q}$. Therefore, $r-2\leq aw(\mathbb{Z}_{2q},3)-1$.
		
		If $q=3^i$ for some $i$, then $n$ can not be a power of $3$ because $7\cdot 3^i +1 \leq n\leq 9\cdot 3^i-1$. Suppose $n=3^m$, then $2q$ is not twice a power of $3$ and clearly $2q$ is not a power of $3$. Therefore, by Lemma~\ref{ZnLessLog} we have 
		$$r\leq aw(\mathbb{Z}_{2q},3)+1\leq \lceil \log_3 (2q)\rceil +2 \leq \lceil \log_3 (2n/7) \rceil +2 = \lceil \log_3 (2\cdot 3^{m}/7) \rceil +2=m+1\leq f(n)-1.$$

		Suppose now that $n\not = 3^m$. If $q=3^i$ for some $i$ then $i\leq m-2$. Otherwise, if $i\geq m-1$ then $q\geq 3^{m-1}\geq 1/7n$ which contradicts the fact that $q<1/7n$. Therefore, $2q\leq 2\cdot 3^{m-2}=18\cdot 3^{m-4}$ and so by induction and Lemma~\ref{ZnLessLog}, $r\leq aw(\mathbb{Z}_{2q},3)+1 = aw([2q],3)+1\leq m+2\leq f(n)-1.$ If $q$ is not a power of $3$, then again using Lemma~\ref{ZnLessLog}, $r\leq aw(\mathbb{Z}_{2q},3)+1 \leq aw([2q],3)$. Notice that $6\cdot 3^{m-3}+2/7\leq 2n/7\leq 18\cdot 3^{m-3}$, and so $aw([2q],3)\leq m+2$ by induction. Therefore, $r\leq m+2\leq f(n)-1$.
		

\end{document}